\documentclass[oneside,11pt]{elsarticle}

\usepackage{amsmath}
\usepackage{amssymb}
\usepackage{amsthm}
\usepackage{tikz}
\usetikzlibrary{matrix,arrows}
\usepackage{graphics,graphicx}
\usepackage{verbatim}
\usepackage{color}
\input{xy}
\xyoption{all}

\usepackage[left=1.4in,top=1.5in,right=1.4in,bottom=1.5in]{geometry}

\input{kmacros3.sty}

\newcommand{\ZZ}{\mathbf Z}

\newcommand{\QQ}{\mathbf Q}

\newcommand{\AC}{AC{ }}
\renewcommand{\bZ}{\ZZ}
\renewcommand{\bQ}{\QQ}
\newcommand{\ie}{{\itshape i.e.} }

\DeclareFontFamily{OT1}{rsfs}{}
\DeclareFontShape{OT1}{rsfs}{n}{it}{<-> rsfs10}{}
\DeclareMathAlphabet{\mathscr}{OT1}{rsfs}{n}{it}

\newtheoremstyle%
{custom}%
{}
{}
{}
{}
{}
{.}
{ }
{\thmname{}
\thmnumber{}%
\thmnote{\bfseries #3}}%

\newtheoremstyle%
{Theorem}%
{}%
{}%
{\itshape}%
{}%
{}%
{.}%
{ }%
{\thmname{\bfseries #1}%
\thmnumber{\;\bfseries #2}%
\thmnote{\;(\bfseries #3)}}%
\newcounter{con}
\theoremstyle{Theorem}
\newtheorem{thm}{Theorem}[section]
\newtheorem{cor}[thm]{Corollary}
\newtheorem{lem}[thm]{Lemma}
\newtheorem*{mainthm}{Main Theorem}

\theoremstyle{definition}
\newtheorem{dff}[thm]{Definition}
\newtheorem{definition}[thm]{Definition}
\newtheorem{xmp}[thm]{Example}
\newtheorem{remark}[thm]{Remark}
\newtheorem{conj}[con]{Conjecture}

\newtheorem*{mainquestion}{Motivating Question}

\def\mbf{\mathbf}

\def\tn{\textnormal}

\def\goth{\mathfrak}

\def\Tr{\tn{Tr}}
\DeclareMathOperator{\AlDiv}{WSh}
\newcommand{\QAlDiv}{\AlDiv_{\bQ}}
\newcommand{\ZPAlDiv}{\AlDiv_{\bZ_{(p)}}}
\def\hst{hereditary surjective trace}

\usepackage{setspace}
\usepackage{calc}

\renewcommand{\O}{\mathcal O}

\journal{Journal of Algebra}

\begin{document}

\begin{frontmatter}

\title{Semi-log canonical vs $F$-pure singularities}
\author[LEM]{Lance Edward Miller}
\address[LEM]{Department of Mathematics\\ University of Utah\\ Salt Lake City, UT, 84112, USA}
\fntext[LEM]{The first author was partially supported by a National Science Foundation VIGRE Grant, \#0602219}
\ead{lmiller@math.utah.edu}
\author[KS]{Karl Schwede}
\address[KS]{Department of Mathematics\\ The Pennsylvania State University\\ University Park, PA, 16802, USA}
\fntext[KS]{The second author was partially supported by the National Science Foundation postdoctoral fellowship \#0703505 and NSF DMS \#1064485/0969145}
\ead{schwede@math.psu.edu}

\begin{keyword}
Frobenius map, Frobenius splitting, F-purity, semi-log canonical, log canonical, inversion of adjunction, normalization, seminormal
\end{keyword}

\begin{abstract}
If $X$ is Frobenius split, then so is its normalization and we explore conditions which imply the converse.  To
do this, we recall that given an $\mathcal{O}_X$-linear map $\phi : F_* \mathcal{O}_X \to \mathcal{O}_X$, it always extends to a map $\bar{\phi}$ on the normalization of $X$.  In this paper, we study when the surjectivity of $\bar{\phi}$
implies the surjectivity of $\phi$.
 While this doesn't occur generally, we show it always happens if certain tameness conditions are satisfied for the normalization map.
Our result has geometric consequences including a connection between
$F$-pure singularities and semi-log canonical singularities, and a more familiar version of the ($F$-)inversion of adjunction formula.
\end{abstract}

\end{frontmatter}


\numberwithin{equation}{thm}

\section{Introduction}

Frobenius splittings of algebraic varieties appear prominently in tight closure theory, have emerged as a fundamental tool in the study of the representation theory of algebraic groups, and have tantalizing links to concepts in the minimal model program.
Suppose that $R$ is a reduced ring of characteristic $p > 0$ with normalization in its total ring of fractions $R^{\textnormal{N}}$.  It follows from \cite[Exercise 1.2.4(E)]{BrionKumarFrobeniusSplitting} that if $\Spec R$ is Frobenius split then so is $\Spec R^{\textnormal{N}}$.  The goal of this paper is to study to what extent the converse holds.  Of course, there are many non-Frobenius split affine varieties whose normalizations are regular and thus Frobenius split (for example, the cusp $y^2 = x^3$ in any characteristic) and so we focus our attention not just on the ring, but on the ring and a choice of a potential Frobenius splitting.

By definition, a \emph{Frobenius splitting} on a variety $X = \Spec R$ is an $\O_X$-linear map $\phi \colon F_* \O_X \to \O_X$ which sends the section $1$ to $1$.  When $X$ is affine, the existence of a Frobenius splitting is equivalent to existence of a map $\phi \colon F_* \O_X \to \O_X$ such that $\phi(z) = 1$ for some section $z \in F_* \O_X$ (in other words, to require that $\phi$ is surjective).

On a normal $X = \Spec R$, surjective maps $\phi : F_* \O_X \to \O_X$ are very closely related to boundary divisors $\Delta$ such that $(X, \Delta)$ is log canonical, see \cite{HaraWatanabeFRegFPure} and \cite{SchwedeFAdjunction}.  Based upon our intuition with non-normal log canonical singularities (which are usually called \emph{semi-log canonical}) one should be able to detect surjective $\phi : F_* \O_X \to \O_X$ by studying the normalization of $X$.  In fact, this correspondence between log canonical pairs $(X, \Delta)$ and surjective $\phi$ suggests that the following question should have an affirmative answer.

\begin{mainquestion}
\label{quest.motivation}
Suppose that $X$ is an affine variety in characteristic $p > 0$ with normalization $\eta \colon X^{\textnormal{N}} \to X$, and set $F^e \colon X \to X$ to be the $e$-iterated Frobenius.  Given a map $\phi \colon F^e_* \O_X \to \O_X$, it always extends to a unique map $\overline{\phi} \colon F^e_* \O_{X^{\textnormal{N}}} \to \O_{X^{\textnormal{N}}}$.  If $\overline{\phi}$ is surjective, does it follow that $\phi$ is also surjective?
\end{mainquestion}

Perhaps unfortunately, this question has a negative answer.  A counterexample is given by the scheme $X = \Spec \bF_2[x,y,z]/(x^2y - z^2)$,
\cite[Example 8.4]{SchwedeFAdjunction}.  However, that counterexample possesses substantial inseparability:  If $C$ is the non-normal locus of $X$ and $B$ is its pre-image inside $X^{\textnormal{N}}$, then the induced map $B \to C$ is generically inseparable.
We show that if we can avoid this inseparability and also a certain variant of wild ramification, then the question has a positive answer.  In particular, we say that a ring possesses \emph{\hst} if it avoids these positive characteristic pathologies (see Definition \ref{def.hst} for a precise definition). Our main theorem follows:

\begin{mainthm}
\label{thm:MainTheorem} \textnormal{[Theorem \ref{thm:main}]}
Suppose that $X$ is a reduced $F$-finite affine scheme having \hst.  Further suppose that $\phi \colon F^e_* \O_X \to \O_X$ is an $\O_X$-linear map.  If the unique extension $\overline{\phi} \colon F^e_* \O_{X^{\textnormal{N}}} \to \O_{X^{\textnormal{N}}}$ is surjective, then $\phi$ is also surjective.
\end{mainthm}

The main theorem should also be viewed as complementary to \cite[Theorem 6.26]{SchwedeTuckerTestIdealFiniteMaps}.  In the context of a finite surjective map $\pi \colon Y \to X$ between \emph{normal} varieties, that result answers the same question for $\phi \colon F_* \O_X \to \O_X$ and its extension $\overline{\phi} \colon F_* \O_Y \to \O_Y$ (if it exists).  Also see \cite{GotoWatanabeTheStructureOfOneDimensionalFPureRings} where 1-dimensional non-normal $F$-pure rings are studied.

We now discuss in more detail the motivation for this theorem.  In this context, maps $\phi \colon F_* ^e\O_{X^{\textnormal{N}}} \to \O_{X^{\textnormal{N}}}$ correspond to $\QQ$-divisors $\Delta$ on $X^{\textnormal{N}}$ such that $K_{X^{\textnormal{N}}} + \Delta$ is $\QQ$-Cartier, see \cite{SchwedeFAdjunction} and Section \ref{sec:prelim}. The condition $\phi$ being surjective (\ie the pair $(X^{\textnormal{N}}, \Delta)$ being $F$-pure) corresponds to  the pair $(X^{\textnormal{N}}, \Delta)$ having log canonical singularities \cite{HaraWatanabeFRegFPure}, \cite{SchwedeFAdjunction}.  On a characteristic zero variety $X$ that is S2 (\ie Serre's second condition) and G1 (\ie Gorenstein in codimension $1$) with normalization $\eta : X^{\textnormal{N}} \to X$ and conductor divisor $B \subseteq X^{\textnormal{N}}$, if there exists a $\bQ$-divisor $\Gamma$ on $X$ such that $K_X + \Gamma$ is $\QQ$-Cartier and furthermore that $(X^{\textnormal{N}}, \eta^* \Gamma + B)$ is log canonical, then $(X, \Gamma)$ is called \emph{semi-log canonical}.  Via the map-($\phi$) divisor-($\Delta$) correspondence, our main theorem should be interpreted as saying:
\vskip 6pt
 \begin{tabular}[c]{ll}
    \parbox[t]{5in}{ \emph{Varieties that are S2, G1 and have $F$-pure singularities are the right characteristic $p > 0$ analog of semi-log canonical varieties in characteristic zero.} }
  \end{tabular}
\vskip 6pt
As an explicit example of this statement, and of the translation between maps $F^e_* \O_X \to \O_X$ and divisors, we have the following corollary of our main theorem, a criterion for certain non-normal algebraic varieties to be $F$-split.
\vskip 6pt
\noindent
{\bf{Corollary \ref{cor.hstNormalizationQGorenstein}.}} {\it
Suppose that $X = \Spec R$ is an affine $F$-finite scheme satisfying \hst{} and which is also S2, G1 and $\bQ$-Gorenstein with index not divisible by the characteristic $p$.  Set $X^{\textnormal{N}}$ to be the normalization of $X$ and set $B$ to be the divisor on $X^{\textnormal{N}}$ corresponding to the conductor ideal, \ie $\goth{c} = \O_{X^{\textnormal{N}}}(-B)$.  Then $X$ is $F$-pure (equivalently $F$-split) if and only if $(X^{\textnormal{N}}, B)$ is $F$-pure.
}
\vskip 6pt
\noindent In the above Corollary, if one replaces the two occurrences of the word $F$-pure by semi-log canonical and log canonical (respectively), one obtains a well known criterion for a scheme having semi-log canonical singularities, see \cite{KollarFlipsAndAbundance}.

More generally, the correspondence between log canonical and $F$-pure singularities is still conjectural, although an area of active research.
\begin{conj}
\label{conj.FpureVsLogCanonical}
Suppose that $Y$ is a normal variety of characteristic zero and that $\Delta$ is a divisor on $Y$ such that $K_Y + \Delta$ is $\QQ$-Cartier.  The pair $(Y, \Delta)$ is log canonical if and only if $(Y, \Delta)$ has dense $F$-pure type\footnote{A characteristic $0$ pair $(Y,\Delta)$ has dense $F$-pure type if the map corresponding to $\Delta$ is surjective after reduction to characteristic $p > 0$ for infinitely many $p > 0$. See \cite{HaraWatanabeFRegFPure} for more details. }.
\end{conj}

See \cite{MustataSrinivasOrdinary} and \cite{MustataOrdinary2} for connections between this conjecture and open conjectures in arithmetic geometry (these connections will not be utilized in this paper).   The setting for our applications will necessarily deal with geometry on non-normal schemes, and as such we will be dealing with an extended notion of divisors, see \cite{HartshorneGeneralizedDivisorsOnGorensteinSchemes, KollarFlipsAndAbundance}. For us, the important divisors are the $\QQ$-almost Cartier divisor (or simply \emph{$\QQ$-\AC divisors}), and we will review the definitions in detail in Section \ref{sec:Divisors}. Our main result implies the following:
\vskip 6pt
\noindent
{\bf{Corollary \ref{cor.FpureVsSLC}}.}
{\it
Assume now that $X$ is an S2, G1 and seminormal variety in characteristic zero.  Further suppose that $\Delta$ is an $\QQ$-\AC divisor (\ie $\QQ$-Weil sheaf) on $X.$
Assuming that Conjecture \ref{conj.FpureVsLogCanonical} holds, $(X, \Delta)$ has semi-log canonical singularities if and only if it has dense $F$-pure type.
}
\vskip 6pt
Another application of Theorem \ref{thm:MainTheorem} is to give a statement of inversion of adjunction for divisors on characteristic $p > 0$ schemes with \hst{} that closely aligns with the characteristic $0$ picture, compare with \cite{KawakitaInversion}.  Inversion of adjunction for characteristic $p > 0$ schemes was studied first in \cite{HaraWatanabeFRegFPure}, see also Takagi \cite{TakagiInversion} and \cite{SchwedeFAdjunction}.  However the direct analog of Kawakita's result is not possible in characteristic $p > 0$ \cite[Example 8.4]{SchwedeFAdjunction}. The culprit is the same conditions of inseparability and wild ramification which obstruct a positive answer to the motivating question.  However:
\vskip 6pt
\noindent
{\bf{Corollary \ref{cor.invofadj}.}}
{\it
Suppose that $X$ is a normal scheme of characteristic $p > 0$, $\Delta$ is a $\QQ$-divisor on $X$ and $S$ is a reduced integral Weil divisor on $X$, with no common components with $\Delta$, such that $K_X + S + \Delta$ is $\QQ$-Cartier with index not divisible by $p$.  Denote by $S^{\textnormal{N}}$ the normalization of $S$ and $\eta \colon S^{\textnormal{N}} \to S \subseteq X$ the natural map. There exists a canonically determined $\QQ$-divisor $\Delta_{S^{\textnormal{N}}}$ on $S^{\textnormal{N}}$ such that $\eta^*( K_X + S + \Delta) \sim K_{S^{\textnormal{N}}} + \Delta_{S^{\textnormal{N}}}$.  Furthermore, if $\O_S$ has \hst, then $(X, S + \Delta)$ is $F$-pure near $S$ if and only if $(S^{\textnormal{N}}, \Delta_{S^{\textnormal{N}}})$ is $F$-pure.
}
\vskip 12pt
\hskip -12pt{\it Acknowledgements:}  The authors would like to thank Bryden Cais, Neil Epstein and Kevin Tucker for inspiring conversations related to this project.  The authors would also like to thank the referee, Shunsuke Takagi and Kevin Tucker for numerous useful comments on an earlier draft of this manuscript.  This work was initiated when the second author was visiting the University of Utah during Fall 2010.

\section{Preliminaries}
\label{sec:prelim}

Throughout this article, all rings will be commutative with unity, Noetherian, excellent and, unless otherwise specified, of prime characteristic and all schemes will be separated and Noetherian.
We frequently use the following notations and conditions on rings and schemes.  Finally, in order to avoid confusion, we remark that we use the notation $\bZ_{(p)}$ to denote the ring of integers $\bZ$ localized at the prime ideal $(p) = p\bZ$.

For a reduced ring $R$ (resp. $X$ is a reduced scheme), we use $K(R)$ (resp. $K(X)$) to denote its total ring of fractions.  Furthermore, we use $R^{\textnormal{N}}$ (resp. $X^{\textnormal{N}}$) to denote the normalization of $R$ in $K(R)$ (resp. the normalization of $X$ in $K(X)$). A ring is called S2 if it satisfies Serre's second condition, \ie the localization at any prime in $R$ of height at least $2$ or $1$ has depth at least $2$ or $1$ respectively.  It is called G1 if it is Gorenstein in codimension 1.  A finite extension of reduced rings $R \subseteq S$ is called \emph{subintegral} (resp. \emph{weakly subintegral}) if it induces a bijection on prime spectra such that the residue field maps are all isomorphisms (resp. purely inseparable).  A reduced ring $R$ is called \emph{seminormal} if it possesses no proper subintegral extensions inside its own field of fractions.  Any reduced ring $R$ has a \emph{seminormalization} $R^{\textnormal{SN}}$ which is the unique largest subintegral extension of $R$ which is contained inside $K(R)$.  For an introduction to seminormalization, see \cite{GrecoTraversoSeminormal} or \cite{VitulliWeakNormalityAndSeminormalitySurvey}.  It is important to note that if $R$ is seminormal then the conductor ideal $\goth{c} = \Ann_R(R^{\textnormal{N}}/R)$ is a radical ideal in both $R$ and $R^{\textnormal{N}}$, see \cite{GrecoTraversoSeminormal} or \cite{HunekeSwansonIntegralClosure}.

For rings of characteristic $p > 0$, we denote by $F \colon R \to R$ the Frobenius homomorphism sending $r$ to its $p$-th power.  Given any $R$-module, $M$, one can view $M$ as an $R$-module via the Frobenius map and obtain a new $R$-module $F_* M$, called the \emph{Frobenius pushforward of $M$}.  This is the $R$-module with the same underlying additive group as $M$ but scalar multiplication is defined by the rule $r \cdot x = r^p x$ for $r \in R$, $x \in M$. One can iterate the Frobenius and for each $e$ we get a Frobenius push forward $F_*^e M$ whose $R$-module structure is defined similarly.

\subsection{Divisors and semi-log canonical singularities}\label{sec:Divisors}

For a scheme $X$ over a field of characteristic $p$, one can ask about the $\O_X$-module structure of $F_*^e \O_X$. Connections between this question and the geometry of $X$ has been a strong guiding force in the study of such characteristic $p > 0$ schemes. For example, $X$ is regular if and only if $F_*^e\O_X$ is a flat $\O_X$-module \cite{KunzCharacterizationsOfRegularLocalRings}. We now review some of these relationships. A scheme $X$ is {\it{F-finite}} provided $F_*^e \O_X$ is coherent, \ie a finitely generated $\O_X$-module for some (equivalently all) $e > 0$.
\smallskip

\noindent
  \begin{tabular}[c]{ll}
    {\scshape Convention:} & \parbox[t]{4.4in}{Throughout this paper, all positive characteristic schemes will be
  assumed to be $F$-finite.}\\
  \end{tabular}
\vskip 3pt
\smallskip

The class of $F$-finite schemes is particularly nice because $F$-finite schemes are abundant, namely varieties over a perfect field are all $F$-finite.  $F$-finite schemes are always locally excellent \cite{KunzOnNoetherianRingsOfCharP} and they always locally have dualizing complexes \cite{Gabber.tStruc}. As in the introduction, we say $X$ is {\it{$F$-split}} provided there is a map $F_* \O_X \to \O_X$ which sends $1$ to $1$ and we say $X$ is {\it{$F$-pure}} provided the Frobenius map on $\O_X$ is a pure morphism.  For $F$-finite affine schemes, $F$-splitting and $F$-purity coincide, see for example \cite{FedderFPureRat} and \cite{HochsterCyclicPurity}.

To describe how these notions relate to characteristic $p > 0$ geometry, we use the setting of pairs.  For a more detailed treatment see \cite{SchwedeFAdjunction,SchwedeTuckerTestIdealFiniteMaps}.  A {\it{prime divisor}} on a normal irreducible scheme $Y$ is a reduced irreducible subscheme of codimension 1 and a {\it{Weil divisor}} is any element of the free Abelian group generated by the prime divisors.  This Abelian group is denoted by $\text{Div}(Y)$. A $\QQ$-divisor is an element of $\tn{Div}(Y) \otimes_\ZZ \QQ$, \ie  a divisor with rational coefficients.  An element $\Delta \in \tn{Div}(Y) \otimes_\ZZ \QQ$ is called an \emph{integral divisor} if it is contained within $\text{Div}(Y)\subseteq \tn{Div}(Y) \otimes_\ZZ \QQ$ (in other words, if it is a Weil divisor, and we wish to emphasize that its coefficients are integral).  A {\it{Cartier divisor}} is a Weil divisor that is locally principal and a $\QQ$-Cartier divisor is a $\bQ$-divisor $D$ such that $mD$ is integral and Cartier for some $0 \neq m \in \ZZ$.  For a $\bQ$-Cartier divisor $D$, the smallest positive integer $m$ such that $mD$ integral and Cartier is called the \emph{index of $D$}.  See \cite{LazarsfeldPositivity2} for additional discussion of $\bQ$-divisors in this context.

We now discuss divisors on non-normal S2 schemes.  One should note that our main theorem can be proven without appeal to these objects, and our main Corollary~\ref{cor.FpureVsSLC} is interesting even when the divisor $\Delta = 0$. Thus the reader not already familiar with divisors on non-normal schemes may wish to skip to Section~\ref{sec.ExtendingNormal} at this point.

\begin{definition} \cite{HartshorneGeneralizedDivisorsOnGorensteinSchemes}, \cite{KollarFlipsAndAbundance}
\label{def.AlmostCartierDivisor}
For an S2 equidimensional reduced scheme $X$, an \emph{\AC divisor} (or \emph{almost Cartier divisor}) is a coherent $\O_X$-module $\sF \subseteq K(X)$ satisfying the following properties
\begin{itemize}
\item[(i)]  $\sF$ is S2.
\item[(ii)]  $\sF_{\eta} \cong \O_{X, \eta}$, abstractly, for all points $\eta \in X$ of codimension 0 or 1.
\end{itemize}
\end{definition}
As was pointed out by the referee, the terminology \emph{almost Cartier} is misleading since any Weil divisor on a normal scheme is almost Cartier (the terminology is taken from \cite{HartshorneGeneralizedDivisorsOnGorensteinSchemes}).  Therefore, in order to avoid this confusion, we instead call almost Cartier divisors by the name \AC divisors.

These divisors form an \emph{additive} group via tensor product up to S2-ification, \cite[Part 2, Section 5.10]{EGAIV}, which we denote by $\AlDiv(X)$.
Given $D \in \AlDiv(X)$, we sometimes use $\O_X(D)$ to denote the sheaf $\sF \subseteq K(X)$ defining $D$.  Note that given any $f \in K(X)$ non-zero at any minimal prime, we use $\divisor(f)$ to denote the element of $\AlDiv(X)$ corresponding to $\sF = {1 \over f} \O_X$.

By a \emph{$\QQ$-\AC divisor} (resp. \emph{$\bZ_{(p)}$-\AC divisor}) we mean an element of $\AlDiv(X) \tensor_\ZZ \QQ =: \QAlDiv(X)$ (resp. $\AlDiv(X) \tensor_\ZZ \bZ_{(p)} =: \ZPAlDiv(X)$).
We say that a divisor $\sF \subseteq K(X)$  of $\AlDiv(X)$ is \emph{effective} if $\O_X \subseteq \sF$ and we say that $\Delta \in \QAlDiv(X)$ (resp. $\ZPAlDiv(X)$) is \emph{effective} if $\Delta = D \tensor \lambda$ for some effective $D \in \AlDiv(X)$ and some $0 \leq \lambda \in \QQ$ (resp. $\bZ_{(p)}$).
We say that two \AC divisors $\sF_1$ and $\sF_2$ are \emph{linearly equivalent} if there is $f$ in $K(X)$ such that $f \sF_1 = \sF_2$.
We call an \AC divisor \emph{Cartier} if the sheaf $\sF$ is a line bundle.
We say that an element $\Delta \in \QAlDiv(X)$ (resp. $\Delta \in \ZPAlDiv(X)$) is \emph{$\bQ$-Cartier} (resp. \emph{$\bZ_{(p)}$-Cartier}) if there exists $0 \neq n \in \bZ$ (resp. $0 \neq n \in \bZ \setminus p\bZ$) such that $n \Delta = C \tensor 1$ for some Cartier divisor $C$ in $\AlDiv(X)$.  We say that two elements $D_1 \tensor \lambda_1, D_2 \tensor \lambda_2 \in \QAlDiv(X)$ (resp. $\ZPAlDiv(X)$) are \emph{$\bQ$-linearly equivalent} (resp. $\bZ_{(p)}$-linearly equivalent), denoted $\sF_1 \tensor \lambda_1 \sim_{\bQ} D_2 \tensor \lambda_2 \Delta_2$ (resp. $ D_1 \tensor \lambda_1 \sim_{\bZ_{(p)}}  D_2 \tensor \lambda_2$), if there exists a non-zero integer $m \in \bZ$ (resp. $\bZ \setminus p\bZ$) such that $m \lambda_i \in \bZ$ and there exists an element $f \in K(X)$ such that $\O_X(m \lambda_1 D_1) \cong f \O_X(m \lambda_2 D_2)$.

\begin{remark}
\label{rem.PathologiesForDivisorsOnNonnormal}
One should note that two distinct \AC divisors of $\AlDiv(X)$ can be identified in $\QAlDiv(X)$, see \cite[Page 172]{KollarFlipsAndAbundance}, and in particular, the natural map $\AlDiv(X) \to \QAlDiv(X)$ is generally not injective unlike the case when $X$ is normal.
\end{remark}

\begin{remark}
Working with $\ZPAlDiv(X)$ is not common.  However, this group behaves much better than $\QAlDiv(X)$ in characteristic $p > 0$ for our purposes, see Theorem \ref{thmDivisorMapCorNonNormal}.  When working on normal varieties, this distinction is less important because $\ZPAlDiv(X)$ as a subset of $\QAlDiv(X)$ is simply the set of divisors whose coefficients do not have $p$ in their denominators.  For non-normal varieties however, the natural map from $\ZPAlDiv(X)$ to $\QAlDiv(X)$ is not generally an injection.
Therefore, when working in characteristic $p > 0$ on \emph{non-normal} varieties, we will generally work with $\ZPAlDiv(X)$.
\end{remark}

Suppose that $X$ is a reduced equidimensional scheme which is S2 and G1.  Further suppose that $X$ possesses a \emph{canonical module $\omega_X$}, \ie a module isomorphic to the first non-zero cohomology of a dualizing complex at each point.  By a canonical divisor $K_X$ we mean any embedding of $\omega_X \subseteq K(X)$ up to multiplication by a unit of $H^0(X, \O_X)$, see \cite{HartshonreGeneralizedDivisorsAndBiliaison} and \cite{HartshorneGeneralizedDivisorsOnGorensteinSchemes}.  Notice that the condition that $X$ is G1 is exactly the condition needed to guarantee Definition \ref{def.AlmostCartierDivisor}(ii).

By a pair, we mean a tuple $(Y,\Delta)$ where $Y$ is a S2 and G1 scheme and $\Delta$ is an element of $\QAlDiv(X)$ or $\ZPAlDiv(X)$.  A pair is called {\it{log $\QQ$-Gorenstein}} (resp. {\it{log $\bZ_{(p)}$-Gorenstein}}) provided $K_Y + \Delta$ is $\QQ$-Cartier (resp. $\bZ_{(p)}$-Cartier).  When $\Delta = 0$ the `log' is omitted and the scheme $Y$ is just called $\QQ$-Gorenstein (resp. $\bZ_{(p)}$-Gorenstein). Now suppose that $Y$ is normal and $\pi \colon \widetilde{Y} \to Y$ is a log resolution, which always exists in characteristic zero by \cite{HironakaResolution}.  Decompose $$K_{\widetilde{Y}} - \pi^*( K_Y + \Delta) = \sum a_i E_i,$$ where $E_i$ is prime and $K_{\widetilde{Y}}$ is a canonical divisor that agrees with $K_Y$ wherever $\pi$ is an isomorphism.  The pair $(Y,\Delta)$ is \emph{log canonical} provided $a_i \geq -1$ for all $i$. 

If $X$ is an S2 but not necessarily normal scheme, denote by $\eta \colon X^{\textnormal{N}} \to X$ its normalization and $B$ the divisor associated to the conductor $\goth{c}$ in $X^{\textnormal{N}}$; \ie $\goth{c} = \O_{X^{\textnormal{N}}}(-B)$.  If additionally $X$ is G1, we say $(X,\Gamma)$ is \emph{semi-log canonical} provided $K_X + \Gamma$ is $\QQ$-Cartier and $(X^{\textnormal{N}}, \Delta = \eta^*\Gamma + B)$ is log canonical.  A more detailed treatment of these concepts in characteristic $0$ can be found in \cite[Chapter 16]{KollarFlipsAndAbundance} or \cite[Chapter 3]{KollarModuliOfSurfacesBook}.

\subsection{$p^{-e}$-linear maps and the map-divisor correspondence}

We call an additive map, $\phi \colon \O_X \to \O_X$, {\it{$p^{-e}$-linear}} provided it locally satisfies $\phi(r^{p^e} \cdot x) = r \phi(x)$. These are identified with the set of $\O_X$-module homomorphisms in $\Hom_{\O_X}(F_*^e \O_X, \O_X)$.  Such maps $\phi$ naturally correspond to effective $\QQ$-divisors $\Delta_{\phi}$ such that $K_X + \Delta_{\phi}$ is $\QQ$-Cartier.  In the normal setting, variants of this correspondence have appeared in many places such as \cite{MehtaRamanathanFrobeniusSplittingAndCohomologyVanishing} and \cite{HaraWatanabeFRegFPure}, this correspondence was recently formalized in \cite{SchwedeFAdjunction}.  However, we need a version of this correspondence in the non-normal setting as well.

First suppose that $X = \Spec R$ is a S2 and G1 reduced and equidimensional affine scheme which is also the spectrum of a semi-local\footnote{The semi-local hypothesis is not needed, see \cite[Remark 9.5]{SchwedeFAdjunction}.  However, we restrict our statements to the semi-local case because they are simpler then.} ring.  The usual arguments still imply that $\sHom_{\O_X}(F^e_* \O_X, \O_X) \cong F^e_* \O_X((1-p^e)K_X)$.  Therefore any section $\phi \in \Hom_{\O_X}(F^e_* \O_X, \O_X)$ which is non-zero at each generic point of $X$, induces an effective \AC divisor $D_{\phi} \sim (1-p^e)K_X$ via \cite[Proposition 2.9]{HartshorneGeneralizedDivisorsOnGorensteinSchemes} and \cite[Remark 2.9]{HartshonreGeneralizedDivisorsAndBiliaison} (for example, choose the embedding of $\O_X((1-p^e)K_X)$ into $K(X)$ which sends $\phi$ to $1$).  We define the divisor $\Delta_{\phi}$ associated to $\phi$ to be the element $D_{\phi} \tensor {1 \over p^e - 1} \in \ZPAlDiv(X)$.

Conversely, suppose we are given an effective $\bZ_{(p)}$-\AC divisor $\Delta$ such that $m \Delta = D \tensor 1$ for some integer $m > 0$ which is not divisible by $p$ and for some $D \in \AlDiv(X)$.  Additionally suppose that $-mK_X \sim D$.  Thus $D$ corresponds to a section $\eta \in \O_X(-m K_X)$, up to multiplication by a unit of $H^0(X, \O_X)$.  Without loss of generality, we may assume that $m = p^e - 1$ for some $e \gg 0$ and so we may view $\eta$ as an element $F^e_* \O_X( (1-p^e) K_X) \cong \Hom_{\O_X}(F^e_* \O_X, \O_X)$. These two observations lead us to the following correspondence.

\begin{thm}
\label{thmDivisorMapCorNonNormal}
For $R$ a semilocal $F$-finite reduced, S2, and G1 ring and set $X = \Spec R$. We have the following bijection of sets:
\[
\left\{ \begin{matrix}\text{Effective $\Delta \in \ZPAlDiv(X)$}\\ \text{ such that} \\ \text{$K_X + \Delta$ is $\bZ_{(p)}$-Cartier} \end{matrix} \right\} \leftrightarrow \left\{ \begin{matrix}\text{For each $e > 0$, elements of }\\ \text{ $\Hom_{\O_X}(F^e_* \O_X, \O_X)$ which are } \\ \text{ non-zero on every irreducible } \\ \text{component of $X$} \end{matrix} \right\} \Bigg/ \sim
\]
Here the equivalence relation $\sim$ on the right is generated by two equivalences.
\begin{itemize}
\item[(i)]  We say that $\phi_1, \phi_2 \in \Hom_{\O_X}(F^e_* \O_X, \O_X)$ are equivalent if there exists a unit $u$ in $H^0(X, F^e_* \O_X)$ such that $\phi_1(u \cdot \blank) = \phi_2(\blank)$.
\item[(ii)]  We say that $\phi \in \Hom_{\O_X}(F^e_* \O_X, \O_X)$ and $\phi^{n} := \phi \circ F^e_* \phi \circ \dots \circ F^{(n-1)e}_* \phi$ in $\Hom_{\O_X}(F^{ne}_*  \O_X, \O_X)$ are equivalent.
\end{itemize}
\end{thm}
\begin{proof}
The proof is more subtle than in the normal case:  if $\phi_1$ and $\phi_2 $ are elements of $\Hom_{\O_X}(F^e_* \O_X, \O_X)$ that determine the same $\bZ_{(p)}$-divisor, it does not imply that one is a unit multiple of the other as is the case for normal rings. The reason for this is that there can be torsion in $\AlDiv(X)$.

We certainly know that every $p^{-e}$-linear map $\phi$ as described above induces an effective $\Delta$ in $\ZPAlDiv(X)$ by the procedure described above.  Furthermore, this procedure is surjective by \cite[Proposition 2.9]{HartshorneGeneralizedDivisorsOnGorensteinSchemes}.  Thus we simply have to show that the two equivalence relations described in (i) and (ii) above are sufficient to induce a bijection.  So suppose that $\phi_1 \colon F^{e_1}_* R \to R$ and $\phi_2 \colon F^{e_2}_* R \to R$ induce the same divisor $\Delta$.  By using (ii), we may assume that $e_1 = e_2 = e$.  Thus we have two sections $\phi_1, \phi_2 \in \Hom_R(F^e_* R, R)$ inducing the same $\bZ_{(p)}$-\AC divisors in $\ZPAlDiv(X) = \AlDiv(X) \tensor_{\bZ} \bZ_{(p)}$, say $D_1 \tensor {1 \over p^e - 1}$ and $D_2 \tensor {1 \over p^e - 1}$ respectively.
It follows that there is an integer $m > 0$ not divisible by $p > 0$ such that $m D_1 = m D_2$.  By making $m$ bigger if necessary, we may take $m = p^{e(d-1)} + p^{e(d-2)} + \dots + p^e + 1$ for some $d > 0$.  As a section, we claim that $\phi_i^d \in \Hom_{\O_X}(F^{ed}_* \O_X, \O_X)$ induces the divisor
\[
D_{\phi_i^d} := (p^{e(d-1)} + p^{e(d-2)} + \dots + p^e + 1) D_{i}
\]
in $\AlDiv(X)$ via \cite[Proposition 2.9]{HartshorneGeneralizedDivisorsOnGorensteinSchemes}.

We now prove this claim.  Since divisors are determined in codimension $1$, we may assume that $X = \Spec R$ where $R$ is a $1$-dimensional Gorenstein ring.  Because $R$ is Gorenstein, by duality for a finite map, $\Hom_R(F^e_* R, R)$ is isomorphic to $F^e_* R$ (abstractly).  Choose $\Phi$ to be an $F^e_* R$-module generator of $\Hom_R(F^e_* R, R)$.  We may thus write $\phi_i(\blank) = \Phi(c_i \cdot \blank)$ for some $c_i \in F^e_* R$.  Notice $D_i = \divisor(c_i)$ which corresponds to the sheaf ${1 \over c_i}R$. It follows that
\[
\phi_i^d\left(\blank\right) = \Phi\left(c_i F^e_* \Phi(c_i F^e_* \dots \Phi(c_i F^e_* \Phi(c_i \cdot \blank)))\right) = \Phi^d\left(c_i^{p^{e(d-1)} + p^{e(d-2)} + \dots + p^e + 1} \cdot \blank \right).
\]
Therefore, because $\Phi^d$ generates $\Hom_R(F^{ed}_* R, R)$ as an $F^{ed}_* R$-module by \cite[Appendix F]{KunzKahlerDifferentials} or \cite[Lemma 3.9]{SchwedeFAdjunction},
\begin{equation}
\label{eqnDivisorToPower}
D_{\phi_i^d} = \divisor(c_i^{p^{e(d-1)} + p^{e(d-2)} + \dots + p^e + 1}) = (p^{e(d-1)} + p^{e(d-2)} + \dots + p^e + 1) D_{i}
\end{equation}
which proves the claim.

Thus $\phi_1^d$ and $\phi_2^d$ agree up to multiplication by a unit and so $\phi_1$ and $\phi_2$ are indeed related by relations (i) and (ii). We also have to verify that if $\phi_1$ and $\phi_2$ are related by conditions (i) or (ii), then they induce the same element of $\ZPAlDiv(X)$.  Certainly condition (i) is harmless.  To check condition (ii), one simply has to tensor equation (\ref{eqnDivisorToPower}) by $1 \over p^{ed - 1}$.
\end{proof}

\begin{remark}
If $R$ is not semilocal but additionally there is a quasi-isomorphism $(F^e)^! \omega_R^{\mydot} \qis \omega_R^{\mydot}$ (which occurs for example, if $R$ is essentially of finite type over a field, or more generally a Gorenstein local ring), then the theorem above still holds as long as one restricts the left-hand-side to
\[
\left\{ \begin{matrix}\text{Effective $\Delta \in \ZPAlDiv(X)$}\\ \text{ such that} \\ \text{$K_X + \Delta \sim_{\bZ_{(p)}} 0$ } \end{matrix} \right\}
\]
Of course, the $(\leftarrow)$ direction of the correspondence in Theorem \ref{thmDivisorMapCorNonNormal} always exists.  Alternately, see \cite[Remark 9.5]{SchwedeFAdjunction}.
\end{remark}

\begin{definition}
Suppose that $R$ is a reduced $F$-finite local ring of characteristic $p > 0$.  Fix a map $\phi \colon F^e_* R \to R$.
The pair $(R, \phi)$ is called \emph{$F$-pure} if $\phi$ is surjective.

Suppose that $X$ is an S2, G1 and reduced $F$-finite scheme and $\Gamma \in \ZPAlDiv(X)$ corresponds at each point $x \in X$ to a map $\phi_x \colon F^e_* \O_{X,x} \to \O_{X,x}$ via Theorem \ref{thmDivisorMapCorNonNormal}.  The pair $(X, \Gamma)$ is called \emph{$F$-pure} if $(\O_{X,x}, \phi_x)$ is $F$-pure for all $x \in X$.
\end{definition}

\begin{remark}
While there is substantial freedom in the choice of $\phi_x$ associated to $\Gamma$ at each point, it is straightforward to verify that if two maps $\phi_x$ and $\phi_x'$ are related via the two conditions (i) and (ii) in Theorem \ref{thmDivisorMapCorNonNormal}, then $\phi_x$ is surjective if and only $\phi_x'$ is surjective.
\end{remark}

\subsection{Extending $p^{-e}$-linear maps along finite morphisms}
\label{sec.ExtendingNormal}

For a finite inclusion of rings $R \subset S$ and an $R$-linear map $\phi \colon F_*^e R \to R$ one can ask when there is a $S$-linear map $\overline{\phi} \colon F_*^e S \to S$ so that $\overline{\phi} |_R  = \phi$. The study of such extensions of maps in the case where $R$ and $S$ are both normal domains is carefully laid out in \cite{SchwedeTuckerTestIdealFiniteMaps}.  However, even in that case, an arbitrary $p^{-e}$-linear map on $R$ need not extend to a $p^{-e}$-linear map on $S$.

\begin{xmp}\cite[Example 3.4]{SchwedeTuckerTestIdealFiniteMaps} Consider for example $p = 3$, $e = 1$ and the inclusion $R = \mbf{F}_3[x^2] \subset \mbf{F}_3[x] = S$. An $R$-basis for $F_*R$ is $\{1,x^2,x^4\}$ and an $S$-basis for $F_*S$ is $\{1,x,x^2\}$. Any map $\phi \colon F_*R \to R$ is defined by the images $\phi(1), \phi(x^2), \phi(x^4)$. A map $\overline{\phi} \colon F_*S \to S$ which extends $\phi$ must agree with $\phi$ on $\{1,x^2,x^4\}$. However, $\phi(x^4) = \overline{\phi}(x^4) = x\overline{\phi}(x)$ since $\overline{\phi}$ is $p^{-e}$-linear and so $\overline{\phi}(x) = \phi(x^4)/x$. Thus when $\phi(x^4)$ is not divisible by $x$, $\phi$ cannot extend.
\end{xmp}

Generic separability is necessary to guarantee extensions of maps \cite[Proposition 5.1]{SchwedeTuckerTestIdealFiniteMaps}. When the inclusion of fraction fields is not generically separable then only the zero map extends. This separability allows one to use the trace map as a vehicle for understanding such extensions in the normal case.  In the next section, where we address out motivating question, we will see that the non-normal setting is even more complicated.

\section{Extending $p^{-e}$-linear maps in non-normal rings}
\label{sec:main}

For an affine variety $X = \Spec R$ denote by $X^{\textnormal{N}} = \Spec R^{\textnormal{N}}$ the normalization. If we set $\goth{c}$ to be the conductor (the largest ideal of $R$ that is also an ideal of $R^{\textnormal{N}}$) the inclusion $R \hookrightarrow R^{\textnormal{N}}$ extends to the following commutative diagram where the inclusions are the obvious ones and the other maps are the natural surjections.

\begin{center}
\begin{tikzpicture}
\node (A) at (-1,1) {$R^{\textnormal{N}}$};
\node (B) at (-1,-1){$R$};
\node (C) at (1,0){$R^{\textnormal{N}}/\goth{c}$};
\node(D) at (1,-2){$R/\goth{c}$};
\path[right hook->] (B) edge (A);
\path[->>] (A) edge (C);
\path[->>] (B) edge (D);
\path[right hook->] (D) edge (C);
\end{tikzpicture}
\end{center}

Recall also the following Lemma.

\begin{lem}\cite[Exercise 1.2.4(E)]{BrionKumarFrobeniusSplitting}
\label{lem:ExtendingMapsUnderNormalization}
Suppose that $R$ is a non-normal reduced ring, $R^{\textnormal{N}}$ is its normalization and $\goth{c}$ is the conductor ideal.  Any $R$-linear map $\phi \colon F^e_* R \to R$ is compatible with $\goth{c}$ (\ie $\phi(F^e_* \goth{c}) \subseteq \goth{c}$) and extends uniquely to an $R^{\textnormal{N}}$-linear map $\overline{\phi} \colon F^e_* R^{\textnormal{N}} \to R^{\textnormal{N}}$.
\end{lem}
\begin{proof}
A proof can be found in \cite[Propositions 7.10, 7.11]{SchwedeCentersOfFPurity}.
\end{proof}

Suppose we are given a map $\phi \colon F^e_* R \to R$ as above.  Lemma \ref{lem:ExtendingMapsUnderNormalization} implies that there exists a map $\overline{\phi} \colon F^e_* R^{\textnormal{N}} \to R^{\textnormal{N}}$ and furthermore that the induced map $\phi_{\goth{c}} \colon F^e_* (R/\goth{c}) \to R/\goth{c}$ extends to an $R^{\textnormal N}/\goth{c}$-linear map $\overline{\phi}_{\goth{c}} \colon F^e_* (R^{\textnormal{N}}/\goth{c}) \to R^{\textnormal{N}}/\goth{c}$.  In order to prove our main theorem, we will relate the surjectivity of $\phi$ to that of $\overline{\phi}$ by studying the surjectivity of $\phi_{\goth{c}}$ verses $\overline{\phi}_{\goth{c}}$.

Following the ideas of \cite{SchwedeTuckerTestIdealFiniteMaps}, it is natural to attempt to apply the trace map (for the inclusion $R/\goth{c} \subseteq R^{\textnormal{N}}/\goth{c}$) to solve this problem.  More generally, for a finite inclusion $A \subset B$ of reduced rings $A$ and $B$ where each minimal prime of $B$ lies over a minimal prime of $A$, one can ask whether the trace map $\Tr \colon \tn{Frac}(B) \to \tn{Frac}(A)$ restricts to a \emph{surjective} map from $B$ to $A$ (here $\Tr \colon \tn{Frac}(B) \to \tn{Frac}(A)$ is defined to be the sum of the individual field trace maps).  If this were always the case in our setting, one would have a \emph{surjective} trace map $\Tr \colon R^{\textnormal{N}}/\goth{c} \to R/\goth{c}$ and one could use this show that surjectivity of $\phi$ directly. However, the next examples show the trace map can fail to be surjective (or even fail to induce a map from $B$ to $A$) for inclusions $A \subset B$.

\begin{xmp}
Suppose that $k$ is a perfect field of characteristic $2$ and consider the ring $R = k[x,y,z]/(xz^2 - y^2) \cong k[a^2, ab, b] \subseteq k[a,b]$.  The conductor ideal is $\goth{c} = (b, ab)$ and so $k[a^2] \cong R/\goth{c} \subseteq R^{\textnormal{N}}/\goth{c} \cong k[a]$ is generically purely inseparable.  Thus the trace map $\Tr \colon R^{\textnormal{N}}/\goth{c} \to R/\goth{c}$ is the zero map (and in particular, not surjective).  Likewise, one can construct similar examples of Galois extensions $R/\goth{c} \subseteq R^{\textnormal{N}}/\goth{c}$ of Dedekind domains which are generically separable (and so the trace map is non-zero) but which have wild ramification and so the trace map is not surjective.
\end{xmp}

\begin{xmp}
Consider the extension $A = k[x^2] \subseteq k[x] = B$ where the characteristic of $k$ is \emph{not} equal to $2$.  The trace map $\Tr \colon k(x) \to k(x^2)$ yields $\Tr(B) = A$.  However, we also have the inclusion $A' = k[(x^4 - x^2), x^2(x^4 - x^2)] \subseteq k[x] = B$ noting that $k[(x^4 - x^2), x^2(x^4 - x^2)]$ and $k[x^2]$ have the same fraction field, $k(x^2)$.  Therefore $\Tr(B) \supsetneq A'$.  By using pushout diagrams of schemes (as in \cite{SchwedeGluing}) one can construct a ring $R$ where $A' = R/\goth{c} \subseteq R^{\textnormal{N}} / \goth{c} = B$.
\end{xmp}

\begin{xmp}
\label{xmp1}
Let $k$ be any field and consider rings $$A = \{ (s,t) \in k[x^2] \oplus k[y] \mid \text{$s$ and $t$ have the same constant term} \}$$ and $$B =  \{ (u,v) \in k[x] \oplus k[y] \mid \text{$u$ and $v$ have the same constant term} \}$$ over $k$ both having a node at the origin.  Consider the normalizations $A^{\textnormal{N}}=k[x^2] \oplus k[y]$ and $B^{\textnormal{N}}=k[x]\oplus k[y]$ respectively.  One can see the conductor of $A$ in $A^{\textnormal{N}}$ is the ideal made up of all pairs $(s,t)$ with zero constant term. Likewise the conductor of $B$ in $B^{\textnormal{N}}$ is the ideal made up of all pairs $(u,v)$ with no constant term.

In this example, we consider the trace on each irreducible component, and then add the resulting trace maps.  Somewhat abusively, we call this sum the ``trace'' also, and denote it by $\Tr$.  Clearly this trace map sends the conductor to the conductor but $\Tr(B)$ is not contained in $A$ because $\Tr(1,1)=(2,1)\notin A$.  
Compare with the question ``Trace map attached to a finite homomorphism of noetherian rings'' of Bryden Cais on {\tt{http://mathoverflow.net}} asked on December 3rd, 2009.
\end{xmp}

In light of these examples, the condition that $\Tr(R^{\textnormal{N}}/\goth{c}) = R/\goth{c}$ is too restrictive.  However, the following (recursive) definition, which is substantially weaker, will be exactly what we want in order to answer our motivating question.

\begin{dff}
\label{def.hst}
Suppose that $R$ is a reduced local ring and $X = \Spec R$.  Define $X^{\textnormal{N}} = \Spec R^{\textnormal{N}}$ to be the normalization with conductor $\goth{c}$.  We set $B \subseteq X^{\textnormal{N}}$ and $C \subseteq X$ to be the subschemes defined by $\goth{c}$ and set $B_{\red}$ and $C_{\red}$ to be the associated reduced subschemes.  We say $X$ has {\it{hereditary surjective trace}} provided that there is some irreducible component $C_i$ of $C_{\red}$ dominated by an irreducible component $B_i$ of $B_{\red}$ such that:
\begin{itemize}
\item[(i)]  the induced trace map $\Tr \colon \O_{B_i^{\textnormal{N}}} \to \O_{C_i^{\textnormal{N}}}$ is surjective and
\item[(ii)]  $C_i$ also has \hst{} (this condition is vacuous if $C_i$ is normal).
\end{itemize}
In the language of commutative algebra, a ring having \emph{\hst} means that there exist minimal associated prime ideals of $\goth{c}$ (which is an ideal of both $R$ and $R^{\textnormal{N}}$),  $\goth{p} \subset R$ and $\goth{q} \subset R^{\textnormal{N}}$ such that $R \cap \goth{q} = \goth{p}$ satisfying (i) \mbox{$\Tr \colon (R^{\textnormal{N}}/\goth{q})^{\textnormal{N}} \to (R/\goth{p})^{\textnormal{N}}$} is surjective and (ii) $R/\goth{p}$ also has \hst.

We say that a (non-local) scheme $X = \Spec R$ has \emph{\hst} if it has \hst{} at every point.
\end{dff}

\begin{remark}
Observe that the dimension of $C$ is strictly smaller than that of $X$, and so the recursive process of Definition \ref{def.hst} will stop after finitely many steps.
\end{remark}

\begin{remark}
\label{rem.StrongerHST}
It would also be natural to require conditions (i) and (ii) for \emph{every} irreducible component $C_i$ of $C_{\red}$ dominated by an irreducible component $B_i$ of $B_{\red}$.  However, we will not need this stronger condition.
\end{remark}

We consider the following special case of our main result whose proof we feel is illuminating.

\begin{xmp}
\label{xmp:basecase}
Recall that any complete one dimensional seminormal ring of characteristic $p > 0$ with algebraically closed residue field $k$ is isomorphic to the completion of the coordinate ring of some set of coordinate axes in $\mathbb{A}_k^{n}$ by \cite{VitulliMulticross}.

Suppose that $R$ is an $F$-finite complete two-dimensional S2 seminormal ring with algebraically closed residue field.  Use $R^{\textnormal{N}}$ to denote the normalization of $R$ and suppose that $\phi \colon F^e_* R^{\textnormal{N}} \to R^{\textnormal{N}}$ is a surjective map which is compatible with the conductor ideal $\goth{c}$.  Because $R$ is S2, $R^{\textnormal{N}} / \goth{c}$ is equidimensional and reduced.  Therefore, $R^{\textnormal{N}} / \goth{c}$ is also $F$-pure and thus seminormal and so it is isomorphic to a direct sum of completions of coordinate rings of coordinate axes in various $\mathbf{A}_k^{n}$.  Set $C$ to be the pull-back of the diagram
\[
\{ R^{\textnormal{N}} \xrightarrow{\alpha} R^{\textnormal{N}}/\goth{c} \xleftarrow{\beta} (R/\goth{c})^{\textnormal{SN}} \}.
\]
In other words, $C$ is the ring $\{ (a, b) \in R^{\tn{N}} \oplus (R/\goth{c})^{\tn{SN}} | \alpha(a) = \beta(b) \}.$
Notice that $R \subseteq C$
is subintegral by construction, see \cite{Ferrand2003,SchwedeGluing}.  Therefore, $R \cong C$ since $R$ is seminormal.  This implies $(R/\goth{c})^{\textnormal{SN}} \cong R / \goth{c}$ and so $\Spec R / \goth{c}$ is also isomorphic to coordinate axes.

Given any component $U$ of $\Spec R/\goth{c}$ and a component $W$ of $\Spec R^{\textnormal{N}}/\goth{c}$ dominating it, suppose that $m = [K(W) : K(U)]$. Provided $p$ does not divide $m$ we see that Definition \ref{def.hst}(i) holds. Part (ii) holds vacuously as the components here are normal.
\end{xmp}

Recall a small lemma about extending maps in local rings.

\begin{lem} \cite[Observation 5.1]{SchwedeFAdjunction}
\label{lem:local}
Let $R$ be a local ring and $I$ a proper ideal. Suppose there is a surjective $R$-linear map $\alpha \colon F_*^e (R/I) \to R/I$ which is the restriction of an $R$-linear map $\beta \colon F_*^e R \to R$.  Then $\beta$ is surjective.
\end{lem}

Finally, we prove our main theorem.

\begin{thm}
\label{thm:main}
Suppose that $X$ is a reduced affine $F$-finite scheme having \hst.  Further suppose that $\phi \colon F^e_* \O_X \to \O_X$ is an $\O_X$-linear map.  If the unique extension $\overline{\phi} \colon F^e_* \O_{X^{\textnormal{N}}} \to \O_{X^{\textnormal{N}}}$ is surjective, then $\phi$ is also surjective.
\end{thm}
\begin{proof}
We proceed by induction, the case where $X$ is zero-dimensional is obvious since then $X = X^{\textnormal{N}}$. Furthermore, the statement is local so we may assume that $X = \Spec R$ where $R$ is a local $F$-finite ring of characteristic $p > 0$ and $X^{\textnormal{N}} = \Spec R^{\textnormal{N}}$. We aim to show that an $R$-linear map $\phi \colon F_*^e R \to R$ which extends to a surjective $R^{\textnormal{N}}$-linear map $\overline{\phi} \colon F_*^e R^{\textnormal{N}}  \to R^{\textnormal{N}}$ is also surjective. Notice that $\goth{c}$ is radical in $R^{\textnormal{N}}$ (and thus also in $R$) because $\overline{\phi}$ is surjective and $\goth{c}$ is $\overline{\phi}$-compatible.  Modulo $\goth{c}$, one has maps $\phi_\goth{c} \colon F_*^e (R/\goth{c}) \to R/\goth{c}$ and $\overline{\phi}_\goth{c} \colon F_*^e (R^{\textnormal{N}}/\goth{c}) \to R^{\textnormal{N}}/\goth{c}$ by Lemma \ref{lem:ExtendingMapsUnderNormalization}.

\begin{center}
\begin{tikzpicture}
\node (A) at (-1,1) {$R^{\textnormal{N}}$};
\node (B) at (-1,-1){$R$};
\node (C) at (1,0){$R^{\textnormal{N}}/\goth{c}$};
\node(D) at (1,-2){$R/\goth{c}$};
\node(FA) at (-3,1){$F_*^e R^{\textnormal{N}}$};
\node(FB) at (-3,-1){$F_*^e R$};
\node(FC) at (3,0){$F_*^e (R^{\textnormal{N}}/\goth{c})$};
\node(FD) at (3,-2){$F_*^e (R/\goth{c})$};
\path[right hook->] (B) edge (A);
\draw[->>] (A) edge (C);
\draw[->>] (B) edge (D);
\draw[right hook->] (D) edge (C);
\draw[->] (FA) edge node[above]{$ \overline{\phi}$ }(A);
\draw[->] (FB) edge node[above]{$ \phi$ }(B);
\draw[right hook->] (FB) edge (FA);
\path[->] (FC) edge node[above]{$ \overline{\phi}_\goth{c}$ }(C);
\path[->] (FD) edge node[above]{$ \phi_\goth{c}$ }(D);
\path[right hook->] (FD) edge (FC);
\end{tikzpicture}
\end{center}

We will examine the behavior of the maps $\overline{\phi}, \phi_\goth{c}, \overline{\phi}_\goth{c}$. When $\overline{\phi}$ is surjective then so too is $\overline{\phi}_\goth{c}$ and one can ask whether the surjectivity of $\overline{\phi}_\goth{c}$ implies surjectivity of $\phi_\goth{c}$.  This is equivalent to the surjectivity of $\phi$ by Lemma \ref{lem:local}.

Since $R$ has \hst, there are components of $\Spec R/\goth{c}$ and $\Spec R^{\textnormal{N}}/\goth{c}$ which have a surjective trace map between their normalizations. Specifically, let $\goth{c} \subset \goth{p} \subset R$ and $\goth{c} \subset \goth{q} \subset R^{\textnormal{N}}$ be such that $\Spec R^{\textnormal{N}}/\goth{q}$ and $\Spec R/\goth{p}$ are the components in question.  Furthermore, note that $R/\goth{p}$ is local and $R^{\textnormal{N}}/\goth{q}$ is semilocal.  We consider the normalization of these rings and, by our \hst{} hypothesis, we know that the trace map $\tn{Tr} \colon (R^{\textnormal{N}}/\goth{q})^{\textnormal{N}} \to (R/\goth{p})^{\textnormal{N}}$ is surjective.

There are induced $R$-linear maps $$\phi_{\goth{p}} \colon F_*^e (R/\goth{p}) \to R/\goth{p} \text{ and } \overline{\phi}_\goth{q} \colon F_*^e (R^{\textnormal{N}}/\goth{q}) \to R^{\textnormal{N}}/\goth{q}$$ since minimal primes of a ring are always compatible.  Furthermore because $\overline{\phi}_\goth{c}$ is surjective so too is $\overline{\phi}_{\goth{q}}$.  So we have the following diagram.

\begin{center}
\begin{tikzpicture}
\node (A) at (-1,1) {$R^{\textnormal{N}}/\goth{c}$};
\node (B) at (-1,-1){$R/\goth{c}$};
\node (C) at (1,0){$R^{\textnormal{N}}/\goth{q}$};
\node(D) at (1,-2){$R/\goth{p}$};
\node(FA) at (-3,1){$F_*^e (R^{\textnormal{N}}/\goth{c})$};
\node(FB) at (-3,-1){$F_*^e (R/\goth{c})$};
\node(FC) at (3,0){$F_*^e (R^{\textnormal{N}}/\goth{q})$};
\node(FD) at (3,-2){$F_*^e (R/\goth{p})$};
\path[right hook->] (B) edge (A);
\draw[->>] (A) edge (C);
\draw[->>] (B) edge (D);
\draw[right hook->] (D) edge (C);
\draw[->>] (FA) edge node[above]{$ \overline{\phi}_\goth{c}$ }(A);
\draw[->] (FB) edge node[above]{$ \phi_\goth{c}$ }(B);
\draw[right hook->] (FB) edge (FA);
\path[->>] (FC) edge node[above]{$ \overline{\phi}_\goth{q}$ }(C);
\path[->] (FD) edge node[above]{$ \phi_\goth{p}$ }(D);
\path[right hook->] (FD) edge (FC);
\end{tikzpicture}
\end{center}

Since $R/\goth{c}$ is local, by Lemma \ref{lem:local}, $\phi_\goth{c}$ is surjective if and only if $\phi_{\goth{p}}$ is.  By Lemma \ref{lem:ExtendingMapsUnderNormalization}, $\phi_\goth{p}$ extends to a map $\widehat{\phi_{\goth{p}}} \colon F_*^e (R/\goth{p})^{\textnormal{N}} \to (R/\goth{p})^{\textnormal{N}}$ and $\overline{\phi}_\goth{q}$ extends to a map $$\widetilde{\phi_{\goth{q}}} \colon F_*^e (R^{\textnormal{N}}/\goth{q})^{\textnormal{N}} \to (R^{\textnormal{N}}/\goth{q})^{\textnormal{N}}.$$  For clarity, the diagram for $\widehat{\phi_{\goth{p}}}$ is below.

\begin{center}
\begin{tikzpicture}
\node (A) at (-2,2){$F_*^e (R/\goth{p})^{\textnormal{N}}$};
\node(B) at (-2,0){$F_*^e R/\goth{p} $};
\node (C) at (2,2) {$(R/\goth{p})^{\textnormal{N}}$};
\node (D) at (2,0){$R/\goth{p}$};
\path[right hook->] (B) edge (A);
\path[->] (A) edge node[above]{$ \widehat{\phi_\goth{p}}$ }(C);
\path[->] (B) edge node[above]{$ \phi_\goth{p}$ }(D);
\path[right hook->] (D) edge (C);
\end{tikzpicture}
\end{center}

Note that $\widetilde{\phi_{\goth{q}}}$ is surjective as $\overline{\phi}_{\goth{q}}$ is. By construction, $R/\goth{p}$ has \hst, and since it is lower dimensional, by our inductive hypothesis it is sufficient to show that $\widehat{\phi_{\goth{p}}}$ is surjective.

\begin{samepage}
Now we are in a situation to use the surjective trace map $\tn{Tr} \colon (R^{\textnormal{N}}/\goth{q})^{\textnormal{N}} \to (R/\goth{p})^{\textnormal{N}}$.
The following commutative diagram shows that $\widehat{\phi_\goth{p}}$ is surjective as the top square is commutative by \cite[Corollary 4.2]{SchwedeTuckerTestIdealFiniteMaps}. See also \cite[Proposition 4.1, Theorem 5.6]{SchwedeTuckerTestIdealFiniteMaps}. This completes the proof.
\begin{center}
\begin{tikzpicture}
\node (FB) at (-1.5,0){$F_*^e (R^{\textnormal{N}}/\goth{q})^{\textnormal{N}}$};
\node (FC) at (-1.5,1.5){$F_*^e (R/\goth{p})^{\textnormal{N}}$};

\node (B) at (1.5,0){$(R^{\textnormal{N}}/\goth{q})^{\textnormal{N}}$};
\node (C) at (1.5,1.5){$(R/\goth{p})^{\textnormal{N}}$};

\draw[->>] (B) edge node[right]{$\tn{Tr}$}(C);

\draw[->>] (FB) edge node[left]{$F_*^e\tn{Tr}$}(FC);


\draw[->>] (FB) edge node[below]{$\widetilde{\phi_\goth{q}}$ }(B);
\draw[->] (FC) edge node[above]{$ \widehat{\phi_\goth{p}}$ }(C);
\end{tikzpicture}
\end{center}
\end{samepage}
\end{proof}

\section{Applications}

\subsection{Semi-log canonical and dense $F$-pure type}

We use the map-divisor correspondence described in Theorem \ref{thmDivisorMapCorNonNormal} to study schemes with \hst{} and semi-log canonical singularities. While $F$-pure schemes are known to have log canonical singularities, there is only a conjectural converse, see Conjecture \ref{conj.FpureVsLogCanonical} in the introduction. We pause to prove a few needed lemmas.

\begin{lem}
\label{lem:red}
Any reduced characteristic $0$ scheme $X$ of finite type over an algebraically closed field of characteristic zero has \hst{} after reduction to characteristic $p \gg 0$.
\end{lem}
\begin{proof}
We refer the reader to \cite{HochsterHunekeTightClosureInEqualCharactersticZero} for a detailed description of the reduction to characteristic $p \gg 0$ process.
We also acknowledge the following abuse of notation, by $p \gg 0$ we technically are referring to an open and Zariski dense set of maximal ideals in $A \supseteq \bZ$, a finitely generated $\bZ$-algebra used in the reduction to characteristic $p \gg 0$ process.  Again, see the aforementioned reference for more details.   See \cite[Theorem 2.3.6]{HochsterHunekeTightClosureInEqualCharactersticZero} for the relevance of the algebraically closed base-field assumption.

Suppose we are given components of the conductor subschemes $B_i \subseteq B_{\red} \subset X^{\textnormal{N}}$ and $C_i \subseteq C_{\red} \subset X$ with an degree $n$ finite morphism $B_i \to C_i$.  Performing reduction to characteristic $p \gg 0$ (in particular $p > n$), the trace map $\O_{B_i}^{\textnormal{N}} \to \O_{C_i}^{\textnormal{N}}$ is clearly surjective.  Continuing recursively, we can require that $p > n_j$ for $n_j$'s determined by a finite set of varieties.  This is certainly enough to guarantee \hst.  In fact, we can guarantee the stronger variant of \hst{} found in Remark \ref{rem.StrongerHST}.
\end{proof}

\begin{lem}
\label{lem.DivOfExtensionsPlusConductor}
Suppose that $X = \Spec R$ is an S2, G1 seminormal scheme with normalization $\eta \colon X^{\textnormal{N}} \to X$.  Further suppose that $\phi \colon F^e_* \O_X \to \O_X$ corresponds to a divisor $\Delta_{\phi}$.  Denote by $\overline{\phi} \colon F^e_* \O_{X^{\textnormal{N}}} \to \O_{X^{\textnormal{N}}}$ the extension of $\overline{\phi}$ to $X^{\textnormal{N}}$. The $\QQ$-divisor on $X^{\textnormal{N}}$ corresponding to $\overline{\phi}$, denoted $\Delta_{\overline{\phi}}$, satisfies
\[
\Delta_{\overline{\phi}} = \eta^* \Delta_{\phi} + B
\]
where $B$ is the divisor corresponding to the conductor on $X^{\textnormal{N}}$,  \ie $\goth{c} = \O_{X^{\textnormal{N}}}(-B)$ (the conductor is pure codimension 1 because $X$ is S2).
\end{lem}
\begin{proof}
First we explain how to pull back divisors via $\eta$.  This pull-back process is completely determined in codimension 1 (which is reasonable, since the divisors are determined in codimension 1).  Thus, suppose that we are given a $\sF \tensor {\lambda} = \Delta \in \ZPAlDiv(X) = \AlDiv(X) \tensor_{\bZ} \bZ_{(p)}$.  By construction, since $\sF$ is \AC, it is easy to pull-back (work outside a set of codimension 2, or see \cite[16.3.5]{KollarFlipsAndAbundance}).  We define $\eta^* \Delta := \eta^* \sF \tensor \lambda$.  It is straightforward to verify that this is well defined.


The statement of the lemma can also be checked in codimension 1 and so we assume that $R$ is 1-dimensional and local.  Write $\Delta_{\phi} = (g) \tensor {1 \over p^e - 1}\in \ZPAlDiv(X)$ for some $g \in K(X)$ (we can do this because $\Delta$ is $\bZ_{(p)}$-\AC).  The pullback of $\Delta_{\phi}$ is then just ${1 \over p^e - 1} \divisor_{X^{\textnormal{N}}}(g)$.  We claim it is sufficient to check the statement when $g = 1$ (so that $\Delta_{\phi} = 0$).  To see this claim, choose $\psi \colon F^e_* \O_X \to \O_X$ such that $\Delta_{\psi}$ is zero (which we can do since we have reduced to the case where $R$ is Gorenstein).  This $\psi$ generates $\Hom_{\O_X}(F_*^e\O_X,\O_X)$ as an $F_*^e\O_X$-module and so $\psi(g \cdot \blank) = \phi(\blank)$ for some $g \in F^e_* R$.  It follows that $\overline{\psi}(g \cdot \blank) = \overline{\phi}(\blank)$.  Thus $\Delta_{\overline{\psi}} + {1 \over p^e -1}\divisor_{X^{\textnormal{N}}}(g) = \Delta_{\overline{\phi}}$.  Therefore, if $\Delta_{\overline{\psi}} = \eta^* \Delta_{\psi} + B$, we can add ${1 \over p^e - 1}\divisor_{X^{\textnormal{N}}}(g) = \eta^* \Delta_{\phi}$ to both sides of the equation, which proves the claim.

In this context, $R$ is Gorenstein and local and $R^{\textnormal{N}}$ is regular and semi-local.  It follows that $\Hom_R(R^{\textnormal{N}}, R)$ is a free $R^{\textnormal{N}}$-module.  Fix $\Phi : R^{\textnormal{N}} \to R$ to be a generator (and assume it sends $1$ to some $c \in \goth{c}$ which generates $\goth{c}$ as an $R^{\textnormal{N}}$-module).  Also notice that the assumptions imply that $\phi$ generates $\Hom_R(F^e_* R, R)$ as an $F^e_* R$-module.   Furthermore, for any $F^e_*(R^{\textnormal{N}})$-module generator $\psi \in \Hom_{R^{\textnormal{N}}}(F^e_* (R^{\textnormal{N}}), R^{\textnormal{N}})$, we know that we have $\Phi \circ \psi = \phi \circ (F^e_* \Phi)$ up to multiplication by a unit (which we can then absorb into $\psi$ obtaining a true equality).

At the level of the field of fractions, $\phi$ and $\overline{\phi}$ are the same map and $\Phi$ is multiplication by $c$.  Thus, since $\Phi \circ \psi = \phi \circ (F^e_* \Phi)$, we have $c \cdot \psi = \phi \cdot c$ again at the level of the field of fractions.  Therefore $\psi(c^{p^e - 1} \cdot \blank) = \overline{\phi}(\blank)$.  This implies that $\Delta_{\overline{\phi}}$ is the divisor of $\goth{c}$ as desired. \end{proof}

\begin{cor}
 \label{cor.hstNormalizationQGorenstein}
Suppose that $X = \Spec R$ is an affine $F$-finite scheme satisfying \hst{} and which is also S2, G1 and $\bZ_{(p)}$-Gorenstein.  Set $X^{\textnormal{N}}$ to be the normalization of $X$ and set $B$ to be the divisor on $X^{\textnormal{N}}$ corresponding to the conductor ideal, \ie $\goth{c} = \O_{X^{\textnormal{N}}}(-B)$.  Then $X$ is $F$-pure if and only if $(X^{\textnormal{N}}, B)$ is $F$-pure.
\end{cor}
\begin{proof}
Since $X$ is $\bZ_{(p)}$-Gorenstein, by working sufficiently locally, we may assume that the zero divisor on $X$ corresponds to a map $\phi : F^e_* \O_X \to \O_X$.  Therefore Lemma \ref{lem.DivOfExtensionsPlusConductor} implies that $\overline{\phi} : F^e_* \O_{X^{\textnormal{N}}} \to \O_{X^{\textnormal{N}}}$ corresponds to the divisor $B$.  An application of Theorem \ref{thm:main} completes the proof.
\end{proof}

\begin{xmp}
 If $X = \Spec R$ is a curve singularity with a node at $x \in X$, then $X^{\textnormal{N}}$ is smooth and the conductor ideal is simply the ideal of $x$.  In particular, if one takes a $F^e_*R$-module generator $\phi \in \Hom_R(F^e_* R, R)$ and extends it to a map $\overline{\phi} \in \Hom_{R^{\textnormal{N}}}(F^e_* R^{\textnormal{N}}, R^{\textnormal{N}})$, the divisor $\Delta_{\phi}$ is zero while the divisor $\Delta_{\overline{\phi}}$ is the divisor of the origin with coefficient $1$.

However, now suppose that a curve $X = \Spec R$ has a cusp singularity at $x \in X$.  Note $X^{\textnormal{N}}$ is still smooth and fix $y$ to be the preimage of $x$.  The conductor ideal is the \emph{square} of the ideal of the point of $y \in X^{\textnormal{N}}$.  In particular, if one takes a $F^e_*R$-module generator $\phi \in \Hom_R(F^e_* R, R)$ and extends it to a map $\overline{\phi} \in \Hom_{R^{\textnormal{N}}}(F^e_* R^{\textnormal{N}}, R^{\textnormal{N}})$, the divisor $\Delta_{\phi}$ is zero while the divisor $\Delta_{\overline{\phi}}$ is the divisor of the origin with coefficient $2$.
\end{xmp}




As before, consider a local non-normal S2 affine reduced scheme $X = \Spec R$ and the natural map $\eta \colon X^{\textnormal{N}} \to X$ on the normalization. Write $X^{\textnormal{N}} = \Spec R^{\textnormal{N}}$ and let $C$ be the subscheme associated to the conductor in $X$  and $B$ the divisor associated to the conductor in $X^{\textnormal{N}}$,  \ie $\goth{c} = \O_{X^{\textnormal{N}}}(-B)$. 


\begin{cor}
\label{cor.FpureVsSLC}
Assume Conjecture \ref{conj.FpureVsLogCanonical} holds.  Suppose that $X$ is a seminormal, S2 and G1 pair of finite type over an algebraically closed field $k$ of characteristic zero and $\Delta \in \QAlDiv(X)$ is such that $K_X + \Delta$ is $\QQ$-Cartier.  The pair $(X, \Delta)$ has semi-log canonical singularities if and only if it has dense $F$-pure type.
\end{cor}
\begin{proof}
We now refer the reader to both \cite{HochsterHunekeTightClosureInEqualCharactersticZero} and \cite{HaraWatanabeFRegFPure} for a detailed description of the reduction to characteristic $p \gg 0$ process in this context.  As before, we also acknowledge the following abuse of notation, by $p \gg 0$ we technically are referring to an open and Zariski dense set of maximal ideals in $A \supseteq \bZ$, a finitely generated $\bZ$-algebra used in the reduction to characteristic $p \gg 0$ process.  Again, see the aforementioned references for more details.

Since the S2 property can be detected by examining the support of finitely many $\Ext$ modules, the reductions to characteristic $p \gg 0$ are also S2.  Likewise because $X$ is G1, there is a subset $Z \subseteq X$ of codimension greater than $1$, such that $X \setminus Z$ is Gorenstein.  Thus the same can be preserved after reduction to characteristic $p \gg 0$.  Finally, since $X$ is seminormal, so are its reductions to characteristic $p \gg 0$, to see this use \cite[Corollary 2.7(vii)]{GrecoTraversoSeminormal}.

Set $\eta \colon X^{\textnormal{N}} \to X$ to be the normalization and fix $C \subseteq X$ and $B \subseteq X^{\textnormal{N}}$ to be the subschemes defined by the conductor, respectively.  We can reduce these schemes and subschemes to characteristic $p \gg 0$ as well.

Assume that $(X, \Delta)$ is semi-log canonical which implies that $(X^{\textnormal{N}}, \eta^* \Delta + B)$ is log canonical.  For a Zariski-dense set of  characteristic $p \gg 0$, by assumption we have that $(X^{\textnormal{N}}_p, \eta_p^* \Delta_p + B_p)$ is $F$-pure.  By working on sufficiently small affine charts, we also assume that $K_X + \Delta \sim_{\bQ} 0$ and the same holds for $K_{X_p} + \Delta_p$ in characteristic $p \gg 0$.

In fact, we may even represent $\Delta_p \in \ZPAlDiv(X_p)$ and assume that it is $\bZ_{(p)}$-Cartier (or $\bZ_{(p)}$-linearly equivalent to zero).  It may also be helpful to the reader to notice that, under either hypothesis, $\eta^* \Delta$ has no common components with $B$ and so pathologies discussed in Remark \ref{rem.PathologiesForDivisorsOnNonnormal} can be avoided by viewing $\Delta$ as an element of $\textnormal{WDiv}(X) \tensor \bQ$, the $\bQ$-Weil divisorial sheaves\footnote{The Weil divisorial sheaves $\textnormal{WDiv}(X)$ are the $\sF \in \AlDiv(X)$ which equal $\O_X$ along the non-normal locus of $X$.  Associated $\bQ$-divisors may be treated more like $\bQ$-divisors on normal varieties (in particular, the subgroup of Weil divisorial sheaves has no torsion).}, see \cite[16.2.1]{KollarFlipsAndAbundance}.

For $p \gg 0$, we may assume that the index of $K_{X_p} + \Delta_p$ is not divisible by the characteristic $p > 0$.  Thus $\Delta_p$ induces a map $\phi \colon F^e_* \O_{X_p} \to \O_{X_p}$, which extends to $\overline{\phi} \colon F^e_* \O_{X^{\textnormal{N}}_p} \to \O_{X^{\textnormal{N}}_p}$.  The divisor associated to $\overline{\phi}$ is thus $\eta^* \Delta_{p} + B_p$ and so $\overline{\phi}$ is surjective by Lemma \ref{lem.DivOfExtensionsPlusConductor}.

Now $X_p$ has \hst{} for $p \gg 0$, and so $\phi$ is surjective for a Zariski-dense set of $p \gg 0$.  This proves the ($\Rightarrow$) direction. Conversely, suppose that $(X, \Delta)$ has dense $F$-pure type.  But again by Lemma \ref{lem.DivOfExtensionsPlusConductor} this implies that $(X^{\textnormal{N}}, \eta^* \Delta + B)$ also has dense $F$-pure type, which implies that $(X^{\textnormal{N}}, \eta^* \Delta + B)$ is log canonical by \cite{HaraWatanabeFRegFPure} and so $(X, \Delta)$ is semi-log canonical by definition.
\end{proof}

\subsection{Inversion of adjunction for schemes with \hst}


We first review the inversion of adjunction statement we are concerned with.  Fix a pair $(X, S + \Delta)$ where $X$ is a normal scheme, $S$ a reduced integral Weil divisor and $\Delta$ an effective $\QQ$-divisor, with no common components with $S$, such that $K_X + \Delta + S$ is $\bQ$-Cartier.  Set $\eta \colon S^{\textnormal{N}} \to S$ to be the normalization of $S$ and recall that there is a canonically defined divisor $\Delta_{S^{\textnormal{N}}}$ called the \emph{different of $\Delta$ on $S^{\textnormal{N}}$} which satisfies $K_{S^{\textnormal{N}}} + \Delta_{S^{\textnormal{N}}} \sim_{\bQ} \eta^*(K_X + \Delta + S)$.  In general, adjunction and inversion of adjunction is the comparison of the singularities of a pair $(X,S+\Delta)$ with the singularities of $(S^{\textnormal{N}}, \Delta_{S^{\textnormal{N}}})$. The implication ``$(X, S + \Delta)$ is log canonical $\Rightarrow$  $(S^{\textnormal{N}}, \Delta_S)$ is log canonical'' is called the \emph{adjunction direction}.  The converse implication (at least near $S$) is known as \emph{inversion of adjunction}; see \cite{KawakitaInversion}, \cite[Chapter 17]{KollarFlipsAndAbundance}.

The direct analog of the adjunction direction is known in characteristic $p > 0$ \cite[Prop. 8.2(iv)]{SchwedeFAdjunction}.  In particular, in characteristic $p > 0$, if additionally the index of $K_X + \Delta + S$ is not divisible by $p > 0$ then there exists a canonically determined divisor $\Delta_{S^{\textnormal{N}}}$, called the \emph{$F$-different}, such that $\eta^*(K_X + S + \Delta) \sim_{\bQ} K_{S^{\textnormal{N}}} + \Delta_{S^{\textnormal{N}}}$ and furthermore if $(X, S + \Delta)$ is $F$-pure, then so is $(S^{\textnormal{N}}, \Delta_{S^{\textnormal{N}}})$.  The $F$-different $\Delta_{S^{\textnormal{N}}}$ is constructed as follows:

We work locally and so may assume that $X = \Spec R$ is the spectrum of a local ring.  By hypothesis, there exists a map $\phi_{S+\Delta} \colon F^e_* \O_X \to \O_X$ corresponding to $S + \Delta$.  The ideal $\O_X(-S)$ is $\phi_{S+\Delta}$-compatible (see for example, \cite[Section 7]{SchwedeFAdjunction}), and so there is an induced map $\phi \colon F^e_* \O_S \to \O_S$.  Now, $S$ is not necessarily normal (or S2 or G1) so it is difficult interpret $\phi$ as a divisor.  However, by Lemma \ref{lem:ExtendingMapsUnderNormalization}, $\phi$ extends to a map $\overline{\phi} \colon F^e_* \O_{S^{\textnormal{N}}} \to  \O_{S^{\textnormal{N}}}$.  Finally, we associate to this map the divisor $\Delta_{S^{\textnormal{N}}} := \Delta_{\overline{\phi}}$.

\begin{remark}
It is an open question whether the different and the $F$-different coincide, see \cite[Remark 7.6]{SchwedeFAdjunction} for some discussion of this question.
\end{remark}

In \cite[Example 8.4]{SchwedeFAdjunction} an example is produced where $(X, S + \Delta)$ is not $F$-pure (\ie $\phi_{S+\Delta}$ is not surjective) but $(S^{\textnormal{N}}, \Delta_{S^{\textnormal{N}}})$ is $F$-pure (\ie $\overline{\phi}$ is surjective).  In other words, inversion of adjunction fails.  This counterexample is loaded with exactly the pathologies that are avoided by rings having \hst{} and the following corollary is easy to prove.

\begin{cor}
\label{cor.invofadj}
Suppose that $X$ is a normal scheme of characteristic $p > 0$, $\Delta$ is a $\QQ$-divisor on $X$ and $S$ is a reduced integral Weil divisor on $X$, with no common components with $\Delta$, such that $K_X + S + \Delta$ is $\QQ$-Cartier with index not divisible by $p$.  Denote by $S^{\textnormal{N}}$ the normalization of $S$ and $\eta \colon S^{\textnormal{N}} \to S \subseteq X$ the natural map. There exists a canonically determined $\QQ$-divisor $\Delta_{S^{\textnormal{N}}}$ on $S^{\textnormal{N}}$ such that $\eta^*( K_X + S + \Delta) \sim_{\bQ} K_{S^{\textnormal{N}}} + \Delta_{S^{\textnormal{N}}}$.  Furthermore, if $S$ has \hst, then $(X, S + \Delta)$ is $F$-pure near $S$ if and only if $(S^{\textnormal{N}}, \Delta_{S^{\textnormal{N}}})$ is $F$-pure.
\end{cor}
\begin{proof}We need only prove the final statement.
Using the notation above, it follows from Lemma \ref{lem:local} that $\phi_{S+\Delta}$ is surjective if and only if $\phi$ is surjective.
By our Main Theorem, $\phi$ is surjective if and only if $\overline{\phi}$ is surjective because $S$ has \hst.
\end{proof}


\def\cprime{$'$} \def\cprime{$'$}
  \def\cfudot#1{\ifmmode\setbox7\hbox{$\accent"5E#1$}\else
  \setbox7\hbox{\accent"5E#1}\penalty 10000\relax\fi\raise 1\ht7
  \hbox{\raise.1ex\hbox to 1\wd7{\hss.\hss}}\penalty 10000 \hskip-1\wd7\penalty
  10000\box7}
\providecommand{\bysame}{\leavevmode\hbox to3em{\hrulefill}\thinspace}
\providecommand{\MR}{\relax\ifhmode\unskip\space\fi MR}
\providecommand{\MRhref}[2]{%
  \href{http://www.ams.org/mathscinet-getitem?mr=#1}{#2}
}
\providecommand{\href}[2]{#2}

\end{document}